\documentclass[a4wide,11pt]{article}
\usepackage[utf8]{inputenc}
\usepackage{amsmath,amssymb,amsthm,mathtools}
\usepackage{graphics} 
\usepackage{graphicx}  
\usepackage{epstopdf}
\usepackage[colorlinks=true]{hyperref}
\usepackage{tikz}
\usetikzlibrary{shapes,arrows}
\usetikzlibrary{positioning}
\usepackage{verbatim}
\usepackage{makeidx}
\usepackage{subfigure}
\usepackage{caption}
\usepackage{setspace}
\usepackage{bbold}
\usepackage{a4wide}
\usepackage{nicefrac}
\usepackage{verbatim}
\usepackage{authblk}
\theoremstyle{plain}
\newtheorem{theorem}{Theorem}[section]

\newtheorem{lemma}[theorem]{Lemma}

\theoremstyle{definition}

\theoremstyle{remark}
\newtheorem{remark}[theorem]{Remark}

\newcommand{\R}{\mathbb{R}}

\newcommand{\Mg}{M_g}

\newcommand{\eps}{\varepsilon}

\def \G{\mathcal{G}}
\def \P{\mathcal{P}}

\title{Uncertainty Quantification in Hierarchical Vehicular Flow Models }
\author{Michael Herty\footnote{Institut f\"{u}r Geometrie und Praktische Mathematik,  		RWTH Aachen University,  
		Templergraben~55, 52062~Aachen, Germany, $\{herty,iacomini\}@igpm.rwth-aachen.de$} \mbox{ and } Elisa Iacomini$^*$}               
\date{}

\begin{document}
\maketitle
\begin{abstract}
We consider kinetic vehicular traffic flow models of BGK type  \cite{MR4063909}. Considering different spatial and temporal scales, those models allow to derive a hierarchy of traffic models including a hydrodynamic description. In this paper, the kinetic BGK--model is extended by introducing a parametric stochastic variable to describe possible  uncertainty in traffic.  The interplay of  uncertainty with the given model hierarchy is studied in detail. Theoretical results on  consistent formulations of the stochastic differential equations on the hydrodynamic level are given. The effect of the possibly negative diffusion in the stochastic hydrodynamic model is studied and numerical simulations of uncertain traffic situations are presented. 
\end{abstract}

{\bf Keywords.} Traffic flow, BGK models, stochastic Galerkin, Aw-Rascle-Zhang model, kinetic equations.

\section{Introduction}

The mathematical description of vehicular traffic flow is possible at different spatial and temporal scales ranging from models for individual cars \cite{FTL1961} up to a description of aggregated quantities like the traffic density \cite{Bando1995,lighthill1955PRSL,MR1951956,MR3231191}. Recent works present models on those scales as well as methods to traverse the existing hierarchy, see e.g.  \cite{MR2834083,aw2000SIAP,MR2580958,MR3356989,MR3541527,MR3721873} and  references therein. We are particularly interested in two scales, the hydrodynamic or  fluid--like models for aggregated quantities and a  statistical description of traffic as e.g. proposed in  \cite{MR2646062,KlarWegener96,BorscheKlar2018,klar1997Enskog}. Our contribution is mainly based on the recently introduced hierarchy \cite{MR4063909} where in particular a class of BGK (Bhatnagar, Gross and Krook~\cite{BGK1954}) models have been considered. The fluid--like models considered are second-order Aw-Rascle-Zhang type models \cite{aw2000SIAP,Zhang2002}. The hierarchy presented in \cite{MR4063909} 
has been  deterministic assuming that all  model parameters  and initial data are  known exactly. However, often there is need to take uncertainties into account, e.g. due to noisy  measurements and due to variations in the behavior of vehicular traffic leading to   uncertainties. Then, it is necessary to extend the concepts to the stochastic case to consider probability laws or  statistical  moments. The treatment of stochastic models can be either non--intrusive, e.g., based on sampling (Monte--Carlo)   \cite{LEcuyer2002,Taimre2011,S18}  or based on collocation \cite{Babuska2007}, or intrusive \cite{O.P.LeMaitre2010,Xiu2010}. In the later approach,  stochastic input is represented by a series of orthogonal functions,  known as generalized polynomial chaos~(gPC) expansions~\cite{S1,S2,S3}, substituted in the governing equations and then projected using a Galerkin projection. We follow this intrusive approach in order to investigate how  uncertainty propagates between the kinetic and the fluid flow hierarchy of description. The possible links are depicted in Figure \ref{fig1}.  Recently,  results using the intrusive approach  for kinetic equations  have been presented and we refer to \cite{K1,K2,Zanella2019,Yuhua2017,Yuhua2018,Carrillo2019,Zanella2020} for corresponding results. For hyperbolic models on the fluid type description there have also been recent results~\cite{H0,H2,H3,H4,S5,GersterJCP2019,GersterHertyCicip2020,KuschMaxPrin2017} -- mostly centered at the question of hyperbolicity of the underlying gPC expanded system of partial differential equations. For the presented investigation we in particular refer to  \cite{gerster2021stability} where a gPC expansion for the Aw-Rascle-Zhang has been established. Therein, it has been shown that for a particular choice of orthogonal functions, the resulting expanded system is hyperbolic, see \cite[Theorem \ref{thm1}]{gerster2021stability}. In this paper we will investigate the link between stochastic BGK and stochastic second order traffic flow models. In \cite{MR4063909} the diffusivity coefficient has been used to classify possible unstable traffic regimes. We will show that the discussion  translates to the stochastic case and allows to characterize possible traffic zones of high risk. Here, we also investigate the dynamic case compared with the previous publication. Our presentation follows the diagram shown in Figure \ref{fig1}, in particular, the indicated blue connections.

The propagation of uncertainty through hierarchies has also been explored e.g. in the case of the Vlasov-Poisson-Fokker-Planck system \cite{Yuhua2018}. Contrary to the approach here, however, the resulting diffusive system has been shown to be well--posed without further assumptions. Due to the nonlinear hyperbolic structure the presented results therein do not extend directly to the present case. Also, in \cite{MR3878149} the propagation of uncertainty is discussed but the origin and treatment of uncertainty is very different to the presented work. 

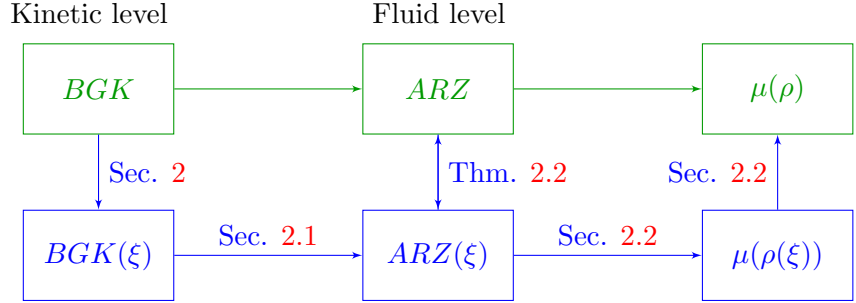
\begin{figure}[htb]\center
\begin{tikzpicture}[auto, node distance=2cm,>=latex']
	\colorlet{darkgreen}{green!60!black}
\node [draw, darkgreen,
	minimum width=2cm,
	minimum height=1.2cm,
	right=1cm 
]  (bgk) {$BGK$};

\node [draw, darkgreen,
     minimum width=2cm, 
	minimum height=1.2cm,
	right=2.5cm of bgk
] (arz) {$ARZ$};

\node [draw, darkgreen,
     minimum width=2cm, 
	minimum height=1.2cm,
	right=2.5cm of arz
] (mu) {$\mu(\rho)$};

\node [draw,blue,
	minimum width=2cm, 
	minimum height=1.2cm, 
	below= 1cm of bgk.south west,
	anchor=north west
]  (bgkuq) {$BGK(\xi)$};

\node [draw,blue,
	minimum width=2cm, 
	minimum height=1.2cm, 
	below= 1cm of arz.south west,
	anchor=north west
]  (arzuq) {$ARZ(\xi)$};

\node [draw,blue,
	minimum width=2cm, 
	minimum height=1.2cm, 
	below= 1cm of mu.south west,
	anchor=north west
]  (muuq) {$\mu(\rho(\xi))$};

\node [
	minimum width=1.6cm, 
	minimum height=0.8cm, 
	above= 0.8cm of bgk.north west,
	left = 0.3cm of bgk.north west,
	anchor=south west
]  (kin) {Kinetic level};

\node [
	minimum width=1.6cm, 
	minimum height=0.8cm, 
	above= 0.8cm of arz.north west,
	left = 0.0cm of arz.north west,
	anchor=south west
]  (mac) {Fluid level};

\draw[darkgreen, ->](bgk) -- (arz) node[midway] {};
\draw[darkgreen, ->](arz) -- (mu) node[midway] {};

\draw[blue,->](bgk) -- (bgkuq) node[midway] {Sec. \ref{sec:BGK}};
\draw[blue,->](arzuq) -- (arz) node[midway] {};
\draw[blue,->](arz) -- (arzuq) node[midway] {Thm. \ref{thm1}};
\draw[blue,->](bgkuq) -- (arzuq) node[midway] {Sec. \ref{sec:S_ARZ}};
\draw[blue,->](muuq) -- (mu) node[midway] {Sec. \ref{sec:mu}};
\draw[blue,->](arzuq) -- (muuq) node[midway] {Sec. \ref{sec:mu}};


\end{tikzpicture}
\caption{Outline of the model hierarchy. The left two columns indicate the kinetic and fluid description of traffic flow as presented in \cite{MR4063909}. The third column refers to the diffusion coefficient $\mu(\rho)$  to classify traffic instabilities. The green hierarchy is deterministic while the blue includes a parametric uncertainty $\xi.$ The indicated links are established in this paper. 
}
\label{fig1}
\end{figure}

\section{Hierarchical stochastic traffic flow models}\label{sec:BGK}

A kinetic traffic flow model reads
\begin{equation}\label{eq:generalKineticModel}
	\partial_t f(t,x,v) + v \partial_x f(t,x,v) = \frac{1}{\varepsilon} Q[f,f](t,x,v),
\end{equation}
where $f(t,x,v)\, : \mathbb{R}^+ \times\mathbb{R} \times [0,V_M]  \rightarrow \mathbb{R}^+$ is the mass distribution function of traffic. The operator $Q$ encodes the detailed car--to--car interactions and it will be modeled in the following as a linear operator of BGK type. The quantity $\varepsilon$ is positive, and yields a relaxation rate weighting the relative strength between the convective term and  source term. The spatial variable is denoted by $x \in \mathbb{R}$ and the velocity $v$ is assumed to be bounded by zero and a maximum speed $V_M.$ Finally, $t\geq 0$ is time and we assume w.l.o.g. that the initial datum $f_0(x,v)$ is such that the  density $\rho_0$ is bounded by one, i.e., 
\begin{equation}
	\int_0^{V_M} f_0(x,v) dv =:\rho_0(x) \leq 1 \; \forall x \in \mathbb{R}. 
\end{equation}
BGK type collision operators prescribe a relaxation to equilibrium at rate $\varepsilon$. In the space homogeneous case, the equilibrium is characterize by a function  $M_f(v;\rho)$ called  Maxwellian possibly depending on $\rho=\rho_0.$  The Maxwellian defines the mean speed of vehicles at equilibrium through the relation
\begin{equation} \label{eq:macroQuantitiesEq}
\quad U(\rho) = \frac{1}{\rho} \int_0^{V_M} v M_f(v;\rho) d v.
\end{equation}
The precise modeling of $Q$ as well as the existence of suitable Maxwellians has been discussed intensively in the literature and we refer e.g. to \cite{PSTV2,MR4063909}. 
Integrating equation~\eqref{eq:generalKineticModel} in velocity space, and provided that $\int_0^{V_M} Q[f,f] dv=0$, and one obtains the evolution equation for the density $\rho(t,x)=\int_0^{V_M} f(t,v,x) dv$ as 
\begin{equation}\label{eq:macroKin}
	\partial_t \rho(t,x)+ \partial_x \left(   \int_0^{V_M} v f(t,x,v) d v \right)= 0 . 
\end{equation}
 If the system approaches equilibrium, $f \rightarrow M_{f}$, then 
\begin{equation}\label{eq:macroEq}
	\partial_t \rho(t,x) + \partial_x \left( \rho(t,x) U(\rho(t,x)) \right) = 0.
\end{equation}
The previous equation and the initial data $\rho(0,x)=\rho_0(x)$   provides a level of description on an aggregated, fluid--like level.   If, however, the system is not at equilibrium, the  equation~\eqref{eq:macroKin} is still coupled to the kinetic equation~\eqref{eq:generalKineticModel}.   In the case $\varepsilon \to  0$,  the interactions of cars are so frequent to instantaneously relax $f$ to the local equilibrium distribution $M_f$.  Instead, we expect that if $\varepsilon>0$ is small but positive,  we are  in a regime where  the kinetic equation is given by an extended continuum hydrodynamic equations as for example  the Aw-Rascle and Zhang model \cite{aw2000SIAP,Zhang2002}. Studying stability properties of traffic patterns in terms of an asymptotic analysis in terms of the parameter  $\varepsilon$ has been conducted e.g. \cite{MR4063909,SeiboldFlynnKasimovRosales2013} using a Chapman-Enskog expansion.
\medskip
The Aw--Rascle--Zhang equations are a system of hyperbolic equations for traffic density $\rho$ and (average) velocity $v=v(t,x)=\frac{q(t,x)}{\rho(t,x)}$ for $\rho>0$. Here, $q(t,x)$ is the flux and the function $h$ will be introduced below in equation \eqref{h}. For some given equilibrium velocity $V^{eq}:\R^+\to\R^+$ decreasing in its argument,  the equations read for $\rho=\rho(t,x)$ and $q=q(t,x)$ with $x\in \R, t \in \R^+:$

\begin{align}\label{eq:ARZ}
&		\partial_t{\rho}+\partial_x({q}-{\rho} h({\rho}))=0, \\
&		\partial_t {q}+\partial_x \Bigl(\frac{{q^2}}{{\rho}} - {q} h({\rho})\Bigr)=\frac{1}{\epsilon}\left( \rho V_{eq}({\rho}) + \rho h(\rho) -{q}\right). \end{align}
In the limit $\epsilon\to 0$ we  formally obtain a  (first--order) consistent approximation of solutions to \eqref{eq:ARZ} to \eqref{eq:macroKin} by defining  for $\rho\geq 0$ 
\begin{equation} V^{eq}(\rho) =U(\rho). \end{equation} 

In the following we are interested in the link between \eqref{eq:ARZ} and \eqref{eq:generalKineticModel}, resp.\ \eqref{eq:BGK} in the {\em stochastic case}.  A key observation  in the deterministic analysis \cite{MR4063909} has been the link between a discretization of the kinetic equation \eqref{eq:generalKineticModel} using a finite number of particles and the Aw--Rascle--Zhang traffic flow model. This connection as been established using the variable $w \in W=[w_{\min}, \infty)$  
\begin{equation}
	w = v+ h(\rho). 
\end{equation}
Here, $w_{\min}=h(0)$ and $h(\rho):\R^+\to\R^+$ is an increasing, differentiable function of the density called  hesitation or pressure function \cite{MR3231191}. We  assume that for $\gamma \in \{1,2\}$  
\begin{equation} \label{h} 
	h(\rho)=\rho^\gamma.
	\end{equation} 
The quantity of $w$ can be understood as a driver's preference that is Lagrangian quantity~ \cite[Section 4]{aw2000SIAP}. Based on a particle description the link 
between the kinetic equation for $g:\R^+\times \R \times W \to \R^+$ 
\begin{equation}\label{eq:BGK}
	\partial_t g(t,x,w)+\partial_x\Big[ (w-h(\rho(t,x))) g(t,x,w)\Big] = \frac{1}{\eps} \Big( M_g(w;\rho(t,x)) - g(t,x,w) \Big),
\end{equation}
and the Aw--Rascle--Zhang equations \eqref{eq:ARZ} for the density $\rho$ and  flux $q$ 
\begin{equation}
	\rho(t,x)=\int_W g(t,x,w) dw,  \qquad q(t,x)=\int_W w \; g(t,x,w) dw
\end{equation}
 has been established using asymptotic analysis in $\varepsilon.$ The Maxwellian $M_g$ can be related to $M_f$, which is assumed to fulfill for any $\rho \in \R$ 
 
 \begin{align}
 	\int_W M_g (w;\rho) \ dw &= \rho,	\label{eq:m1}\tag{M1}\\
 	\int_W w \ M_g(w;\rho)\ dw &=\rho V_{eq}(\rho)+\rho h(\rho).\label{eq:m2}\tag{M2}
 \end{align}
 The function  $V_{eq}:\R \to \R^+_0$ is the previously introduced equilibrium velocity.  
As discussed we are interested in the description of vehicular traffic on the kinetic \eqref{eq:BGK} and fluid--dynamic \eqref{eq:ARZ} level in the presence of parametric uncertainty $\xi.$ This uncertainty may have many origins but for now we simply assume that it can be described by a (possibly multi-dimensional) random variable $\omega$. Let 
the random variable $\omega$ be defined on the probability space $(\Omega_\omega,\mathcal{F}(\Omega),\mathbb{P})$. Further, we denote by  $\xi=\xi(\omega):\Omega_\omega \to \Omega \subset \R^d$ a (possibly d-dimensional) real-valued random variable. Assume further that $\xi$ is absolutely continuous with respect to the Lebesgue measure on $\R^d$ and denote by     $f_\Xi(\xi):\Omega \to \R_+$ the probability density function of $\xi$. For simplicity we assume that the uncertainty enters only in the initial data $g_0(x,w,\omega)$ that is now random field defined on $\R \times W \times \Omega_\omega$ that is  denoted by $g_0(x,w,\xi)=g_0(x,w, \xi(w)): \R\times W \times \Omega \to \R$. Further, we assume that $g_0(x,w,\cdot) \in L^2(\R^d,f_\Xi)$ a.e. in $(x,w).$ Then, we are interested in the evolution of the random field $g(t,x,w,\xi):\R^+\times \R \times W \times \Omega$ governed by a BGK--kinetic equation \eqref{eq:BGK} with uncertain initial data $g_0.$ For the following derivations it is sufficient to assume that first and second moment $g$ w.r.t. to $w$ exist as well as up to second moment in $\xi$.  Further, the derivation is based on the assumption that the  random field $g$ fulfills \eqref{eq:BGK-stoch} pointwise a.e. in $(t,x,w)$ as well as $f_\Xi$ a.s. in $\xi:$ 

\begin{align}\label{eq:BGK-stoch}
&	\partial_t g(t,x,w,\xi)+\partial_x\Big[ (w-h(\rho(t,x,\xi))) g(t,x,w,\xi)\Big] = \frac{1}{\eps} \Big( M_g(w;\rho(t,x,\xi)) - g(t,x,w,\xi) \Big), \\
&	g(0,x,w,\xi) = g_0(x,w,\xi), \\
&	\rho(t,x,\xi)=\int_W g(t,x,w,\xi) dw. 
\end{align}
Next, we turn to the description of the intrusive approach in order to establish the hierarchy indicated in Figure \ref{fig1}.  A random field $g(t,x,w,\cdot) \in L^2(\Omega, f_\Xi)$ can be expressed by a spectral expansion  \cite{S2}
\begin{equation}\label{eq:gPC1}
	g(t,x,w,\xi)=\sum_{k=0}^{\infty}\Tilde{g}_i(t,x,w)\phi_i(\xi),
\end{equation}
where $\phi_i \in L^2(\Omega, f_\Xi)$ are basis functions, typically chose orthonormal with respect to the weighted scalar product, and $\{ \Tilde{g}_i(t,x,w) \}_{i=0}^\infty$ is a set of coefficients: 
\begin{equation}\label{eq:gPC0}
\Tilde{g}_i(t,x,w) = \int_\Omega g(t,x,w,\xi) \phi_i(\xi) f_\Xi(\xi) d\xi.
\end{equation}
 The previous expansion is truncated at $K$ to obtain an approximation with $K+1$ moments.  The projection of $g(t,x,w,\cdot)$ to the span of the $K+1$ base functions is denoted by 
 \begin{equation}\label{projection}
 	G_K( g(t,x,w,\cdot) )(\xi) := \sum\limits_{i=0}^K \Tilde{g}_i(t,x,w) \phi_i(\xi) \; a.s. \xi \in \Omega.
 \end{equation}
 The expansion \eqref{eq:gPC1} is called generalized polynomial chaos expansion (gPC). In particular, for kinetic equations, also more involved than the given BGK equation, this has been explored recently in a series of paper, see e.g. \cite{K1,K2,Zanella2019,Yuhua2017,Yuhua2018,Carrillo2019,Zanella2020}. Therein, also conditions on $\{g_{0,i} \}_{i=0}^\infty$ have been developed to allow for existence of a (weak) stochastic solution $g$. 
 
Next, we establish the connection between the random BGK model \eqref{eq:BGK} and the stochastic Aw--Rascle--Zhang system. Assume $g$ is a pointwise a.e. and integrable solution to the system \eqref{eq:BGK-stoch}. Then, the density $\rho$ and flux $q$ allow for gPC expansion for all $i\in \mathbb{N}:$  

\begin{align}\label{def:rho}
    \rho(t,x,\xi)=\int_W g(t,x,w,\xi)dw = \sum_{i=0}^{\infty} \Tilde{\rho}_i\phi_i(\xi), \;  \Tilde{\rho}_i=\Tilde{\rho}_i(t,x)=\int_W \Tilde{g}_i(t,x,w) dw, \\
      q(t,x,\xi)=\int_W wg(t,x,w,\xi)dw = \sum_{i=0}^{\infty}  \Tilde{q}_i\phi_i(\xi),  \;  \Tilde{q}_i=\Tilde{q}_i(t,x)=\int_W w \Tilde{g}_i(t,x,w) dw.  \label{def:q}
\end{align}
As in \cite{S5,S16,GersterJCP2019} we introduce the Galerkin production for any finite $K>0$ and any  ${u}, {z} \in L^2(\Omega,f_\Xi)$, 
 $\Tilde{u}=(\Tilde{u}_i)_{i=0}^K$, $\Tilde{z}:=(\Tilde{z}_i)_{i=0}^K$ and for all  $i,j,\ell=0,\dots,K:$
 
\begin{align*}
	\G_K[u, z](t,x;\xi) 
	&:=
	\sum_{k=0}^{K} ( \Tilde{u} \ast \Tilde{z} )_k(t,x) \phi_k(\xi), \\
	(\tilde{u} \ast \Tilde{z})_k(t,x) 
	&:=
	\sum_{i,j=0}^{K} \Tilde{u}_i(t,x) \tilde{z}_j(t,x) \mathcal{M}_\ell, \\
	 \left( \mathcal{M}_\ell \right)_{i,j} &:= \int_\Omega \phi_i(\xi) \phi_j(\xi)\phi_\ell(\xi) f_\Xi(\xi) d\xi. 
\end{align*}
Note that $\mathcal{M}_\ell$ is a symmetric matrix of dimension $(K+1) \times (K+1)$ for any fixed $\ell \in \{0,\dots,K\}.$ The Galerkin product $\G_K$ is not the only possible projection of the product of random variables $u,z$ on the subspace  $span\{ \phi_0, \dots, \phi_K \}.$ However, this choice (and additional assumptions on the base functions) have shown to be sufficient to guarantee hyperbolicity of the $p-$system \cite{GersterJCP2019} as well as the Aw--Rascle--Zhang system \cite{gerster2021stability}. Furthermore, we have   ${\Tilde{u} \ast \Tilde{z} = \P(\Tilde{u}) \Tilde{z}}$ for $\P \in \R^{K+1 \times K+1}$ and $\Tilde{u} \in \R^{K+1}$ defined by 
\begin{equation}\label{DefPalpha}
	\P(\Tilde{u})\coloneqq\sum\limits_{\ell=0}^K \Tilde{u}_\ell\mathcal{M}_\ell.	
\end{equation}
The Galerkin product is symmetric, but not associative ~\cite{S15,S4,S18}. Finally, we assume that the chosen  functions $\{ \phi_i \}_i$ fulfill the following properties \cite[A1-A3]{GersterHertyCicip2020}
	\begin{itemize}
	\item[\textup{(A1)}]
	The  matrices $\mathcal{M}_\ell$ and $\mathcal{M}_k$ commute for all ${\ell,k = 0,\ldots,K}$.
	\item[\textup{(A2)}]
	The matrices~$\P(\widehat{u})$ and $\P(\Tilde{z})$ commute for all ${\Tilde{u}, \Tilde{z} \in \mathbb{R}^{K+1}}$.
	\item[\textup{(A3)}]
	There is an eigenvalue decomposition~${
		\P(\Tilde{u}) = V \mathcal{D}(\Tilde{u}) V^{T}
	}$ 
	with constant eigenvectors $V$.	
\end{itemize}
It has been shown that for example the one--dimensional Wiener--Haar basis and piecewise linear multiwavelets fulfill the previous assumptions, but, Legendre and Hermite polynomials do {\em not} fulfill those requirements. 

\par 

Similar to \cite{K1,K2} and for any fixed $K$ we derive a system of equations for the evolution of $\Tilde{g}_i(t,x,w):\R^+\times\R\times W\to \R$   for $i=0,\dots,K$ by projection the operators of equation \eqref{eq:BGK-stoch} to the space $span\{ \phi_i: i=0,\dots,K \}$. 
We further assume that the set of base functions is orthonormal and fulfills the assumptions (A1)--(A3).

\begin{align}\label{eq:UBGK}
&	\partial_t \ \Tilde{g}_i(t,x,w) + \partial_x   \Bigl( \Bigl( w Id -   \P\left(h(\Tilde{\rho}\left(t,x\right))\right) \Bigr) \Tilde{g}(t,x,w) \Bigr)_i  
	=\frac{1}{\eps} \left( \widetilde{M}_i\left(w;\Tilde{\rho}(t,x) \right) -  \Tilde{g}_i(t,x,w) \right), \\
&	\Tilde{g}_i(0,x,w) = \int_\Omega g_{0}(t,x,w,\xi) \phi_i(\xi) f_\Xi(\xi) d\xi
\end{align}
In the derivation of the previous system \eqref{eq:UBGK} we have used the following  results:  
Under  assumptions (A1)-(A3)  $h$, as given by equation \eqref{h}, fulfills   \cite{gerster2021stability}:

\begin{align}
\sum_{j=0}^K \int_\Omega h\left( \sum_{\ell=0}^K \Tilde{\rho}_\ell \phi_\ell(\xi) 
\right)	\Tilde{g}_j \phi_j(\xi) \phi_i(\xi) f_\Xi(\xi) d\xi = 
\left( \P(h(\Tilde{\rho})) \Tilde{g} \right)_i, \; \forall i=0,\dots,K.
\end{align}
Further, we define for $i=0,\dots,K$ 

\begin{align}
	\widetilde{M}_{i}\left(w;\Tilde{\rho}(t,x) \right):=&\int_\Omega  M_g\Bigl(w;\sum_{\ell=0}^K \Tilde{\rho}_\ell(t,x) \phi_\ell(\xi) \Bigr)\phi_i(\xi)  f_\Xi dw d\xi, 	\label{eq:mg} \\ 
	\widehat{M}_g(w,\Tilde{\rho}(t,x),\xi)  :=&  M_g\left( w;\sum_{i=0}^K\tilde{\rho}_i(t,x)\phi_i(\xi)\right).  
	\label{eq:mi}
\end{align}

\subsection{Derivation of Stochastic Aw--Rascle--Zhang Model}\label{sec:S_ARZ}

In \cite{MR4063909} a connection between  two levels of description, i.e., \eqref{eq:BGK} and \eqref{eq:ARZ} has been established under the assumption that the Maxwellian fulfills \eqref{eq:m1} and \eqref{eq:m2}. The next lemma shows that those assumptions extend to directly to the stochastic case. 

\begin{lemma}\label{lemma1}
  Let $K>0.$ Consider a base functions $\phi_i$ and $i=0,\dots,K$ fulfilling (A1)--(A3). Furthermore, assume that the functions $M_g, V^{eq}$ fulfill the assumptions \eqref{eq:m1}-\eqref{eq:m2}. Let $g$ be expanded in a gPC series with $K+1$ modes as given by equation \eqref{eq:gPC1}.   Then, $\widetilde{\Mg}$  defined by \eqref{eq:mi}  fulfill for any $i=0,\dots,K$, $t \geq0,$ and $x\in\R:$
  
  \begin{align}
&     \int_W \widetilde{M}_i\left(w;\Tilde{\rho}(t,x) \right) \ dw = \Tilde{\rho}_i(t,x),	\label{eq:coeff_um1}\tag{UM1}\\
&     \int_W w \ \widetilde{M}_i\left(w;\Tilde{\rho}(t,x) \right) dw = \Bigl( \P (V_{eq}(\Tilde{\rho}(t,x)))\Tilde{\rho}(t,x)+\P (h(\Tilde{\rho}(t,x)))\Tilde{\rho}(t,x)\Bigr)_i.\label{eq:coeff_um2}\tag{UM2}
\end{align} \end{lemma}

\begin{proof}
Due to  \eqref{eq:m1}-\eqref{eq:m2} and \eqref{def:rho} we obtain  for a.e. $(t,x,\xi) \in \R^+\times\R\times\Omega.$

\begin{align}
	\int_W M_g (w;\rho(t,x,\xi)) \ dw &= \rho(t,x,\xi),	\label{eq:um1}\\
	\int_W w \ M_g(w;\rho(t,x,\xi))\ dw &=\rho(t,x,\xi) V_{eq}(\rho(t,x,\xi))+\rho(t,x,\xi) h(\rho(t,x,\xi)).\label{eq:um2}
\end{align}
Integration with respect to $dw$  yields 

\begin{align*}
	\int_W \widetilde{{M}_i}(w,\Tilde{\rho}(t,x)) \ dw = 
	\int_\Omega \int_W {M}_g\Bigl(w; \sum_{j=0}^K \Tilde{\rho}_j(t,x) \phi_j(\xi) \Bigr) 
	\phi_i(\xi) dw f_\Xi(\xi) d\xi= \\ \int_\Omega \sum_{j=0}^K \Tilde{\rho}_j(t,x) \phi_j(\xi) \phi_i(\xi) f_\Xi(\xi) d\xi=  \Tilde{\rho}_i(t,x). 
\end{align*}
The similar computation yielding \eqref{eq:coeff_um2} is omitted.
\end{proof}

In the following result we derive a gPC formulation of the fluid model obtained  by the stochastic BGK model \eqref{eq:BGK-stoch}. Further, we compare this model with the stochastic Aw--Rascle--Zhang model derived in \cite{gerster2021stability}. The theorem shows that under assumption \eqref{eq:ass_thm1} the derived gPC  model is equivalent to the stochastic model of \cite{gerster2021stability}. Therein, it has also been shown that the partial differential equation is hyperbolic.  

\begin{theorem}\label{thm1}
Let $K>0, \epsilon>0$. Assume the base functions $\{ \phi_0, \dots, \phi_K \}$ fulfill (A1)--(A3) and assume that the functions $M_g, V^{eq}$ fulfill the assumptions \eqref{eq:m1}-\eqref{eq:m2} and let $h(\cdot)$ fulfill \eqref{h}. Let $\Tilde{g}_i$ be a strong solution to \eqref{eq:UBGK} and \eqref{eq:mg} for $i=0,\dots,K.$ 
\medskip
Further, assume that for $i=0,\dots,K$ and $(t,x) \in \R^+\times \R$  
\begin{equation}\label{eq:ass_thm1}
	\int_W w^2 \; \Tilde{g}_i(t,x,w) dw= (\P(\Tilde{q}(t,x))\P^{-1}(\Tilde{\rho}(t,x))\Tilde{q}(t,x))_i,
\end{equation}
where $(\Tilde{\rho},\Tilde{q})_i$ are the first and second moment of $\Tilde{g}_i$ as in \eqref{def:rho}--\eqref{def:q} and $\P$ is defined by \eqref{DefPalpha}. 
\par 
Then, the functions $(\Tilde{\rho},\Tilde{q})$ formally fulfill pointwise in $(t,x)\in\R^+\times \R$ and for all $i=0,\dots,K$ the second--order traffic flow model 

  \begin{subequations} \label{eq:UARZ}
  \begin{align}
   &   \partial_t \Tilde{\rho}_i(t,x) + \partial_x \left[\Tilde{q}_i(t,x)- (\P(\Tilde{\rho}(t,x))\Tilde{\rho}(t,x))_i\right]=0\\
   &   \partial_t \Tilde{q}_i(t,x) + \partial_x \left[ (\P(\Tilde{q}(t,x))\P^{-1}(\Tilde{\rho}(t,x))\Tilde{q}(t,x))_i - (\P(\Tilde{\rho}(t,x))\Tilde{q}(t,x))_i \right]= \\
	& \quad \frac1\epsilon \Bigl( \Bigl( \P (V_{eq}(\Tilde{\rho}(t,x)))\Tilde{\rho}(t,x)+\P (h(\Tilde{\rho}(t,x)))\Tilde{\rho}(t,x)\Bigr)_i  - \Tilde{q}_i(t,x) \Bigr) \\
   & \Tilde{\rho}_i(0,x) = \int_W \Tilde{g}_{0,i}(t,x,w) dw, \\
   & \Tilde{q}_i(0,x) = \int_W w \; \Tilde{g}_{0,i}(t,x,w) dw.  
 \end{align}\end{subequations}
\par  
The system \eqref{eq:UARZ} is hyperbolic for $\Tilde{\rho}_i>0.$
\par 
Let the  random fields $(\rho,q)=(\rho,q)(t,x,\xi):\R^+\times\R\times\Omega \to \R^2$  be a pointwise a.e.\ solution with second moments w.r.t.\ to $\xi$ of  the stochastic Aw--Rascle--Zhang system with random initial data: 

  \begin{subequations} \label{eq:UARZ2}
	\begin{align}
		&		\partial_t{\rho}+\partial_x({q}-{\rho} h({\rho}))=0, \\
		&		\partial_t {q}+\partial_x \Bigl(\frac{{q^2}}{{\rho}} - {q} h({\rho})\Bigr)=\frac{1}{\epsilon}\left( \rho V_{eq}({\rho}) + \rho h(\rho) -{q}\right), \\
		&\rho(0,x,\xi)=\rho_0(x,\xi), \;  q(0,x,\xi)=q_0(x,\xi). 
 \end{align}\end{subequations}
Under the previous assumptions on the base functions $\{ \phi_0, \dots, \phi_K\}$ and provided that for all $i=0,\dots,K$ 

\begin{align}\label{cond1}
	\int_\Omega\rho_0(x,\xi) \phi_i(\xi) f_\Xi d\xi = \int_W\Tilde{g}_{0,i}(t,x,w) dw, \; 
	\int_\Omega q_0(x,\xi) \phi_i(\xi) f_\Xi d\xi = \int_W w \Tilde{g}_{0,i}(t,x,w) dw, 
\end{align} 
we have 

\begin{align}\label{thm:final}
	G_K\left(\rho(t,x,\cdot) \right)(\xi)=\sum\limits_{i=0}^K \Tilde{\rho}_i(t,x)\phi_i(\xi) \mbox{ and }  G_K\left(q(t,x,\cdot)\right)(\xi)= \sum\limits_{i=0}^K \Tilde{q}_i(t,x)\phi_i(\xi),
\end{align}
where $(\Tilde{\rho},\Tilde{q})$ fulfill equation \eqref{eq:UARZ}. 
\end{theorem}

Some remarks are in order. The assumption \eqref{eq:ass_thm1} is a closure relation and has been presented in the deterministic case  \cite{MR4063909}.  The result on hyperbolicity of the system \eqref{eq:UARZ} has been presented in \cite{gerster2021stability}. Therein, also the  system for the coefficients $\Tilde{\rho},\Tilde{q}$ of a gPC expansion of the stochastic case of \eqref{eq:ARZ} has been derived, i.e., 
the system \eqref{eq:UARZ}. Condition \eqref{cond1} states the consistency of  initial data of both systems.   
 
\begin{proof}
	The proof is similar to \cite{MR4063909} and given here for completeness. For a pointwise a.e. solution $\Tilde{g}$ and corresponding densities $\Tilde{\rho}$ and fluxes $\Tilde{q}$ according to \eqref{def:rho}--\eqref{def:q}
we obtain for each $i\in \{0,\dots,K\}$ 	by  \eqref{eq:UBGK} and  after integration  on $W$ 

\begin{align}
    \partial_t &  \int_W   \Tilde{g}_i(t,x,w)  dw  + \partial_x    \int_W w \Tilde{g}_i(t,x,w) - \Bigl(\P(h(\Tilde{\rho}(t,x)))\Tilde{g}(t,x,w) \Bigr)_i  dw \\
    &=\frac{1}{\eps} \left( \int_W \Tilde{M}_i(w;\Tilde{\rho}(t,x)) - \Tilde{g}_i(t,x,w) dw \right). 
\end{align}
Since $\Tilde{g} \to \P(\Tilde{\rho})\Tilde{g}$ is linear and by equation \eqref{eq:coeff_um1} of Lemma \ref{lemma1} 

\begin{align} 
	\partial_t \Tilde{\rho}_i(t,x)  + \partial_x \left(\Tilde{q}_i(t,x) - \Bigl(\P(h(\Tilde{\rho}(t,x)))\Tilde{\rho}(t,x) \Bigr)_i \right) = 0.
\end{align}
Furthermore, we integrate \eqref{eq:UBGK}   w.r.t.\ to $w\, dw$ on $W$ to obtain 

\begin{align}
	\partial_t &  \int_W  w \Tilde{g}_i(t,x,w)  dw  + \partial_x    \int_W w^2 \Tilde{g}_i(t,x,w) - w \Bigl(\P(h(\Tilde{\rho}(t,x)))\Tilde{g}(t,x,w) \Bigr)_i  dw \\
	&=\frac{1}{\eps} \left( \int_W w \widetilde{M}_i(w;\Tilde{\rho}(t,x)) - w \Tilde{g}_i(t,x,w) dw \right). 
\end{align}
This yields 

\begin{align}
	\partial_t &  \Tilde{q}_i(t,x)  + \partial_x    \int_W w^2 \Tilde{g}_i(t,x,w) dw -  \Bigl(\P(h(\Tilde{\rho}(t,x)))\Tilde{q}(t,x) \Bigr)_i  \\
	&=\frac{1}{\eps} \left( \int_W w \widetilde{M}_i(w;\Tilde{\rho}(t,x)) dw -  \Tilde{q}_i(t,x) \right). 
\end{align}
Using now \eqref{eq:ass_thm1} and equation \eqref{eq:um2} of Lemma \ref{lemma1} we obtain the momentum equation of the second--order traffic flow model \eqref{eq:UARZ}. 
\par 
Under the assumptions (A1)--(A3) we obtain \eqref{eq:UARZ} is  hyperbolic as proven  in \cite{gerster2021stability}. Therein, also the assertion \eqref{thm:final} has been established. 
\end{proof}

\begin{remark} 
Introducing stochasticity also allows for more general Maxwellians. In particular, the Maxwellian $M_g$ could also depend on $\xi$ directly. Hence, we may assume that 

\begin{align} 
	 M_g(w,\xi;\rho) := {M}(w;{\rho}(t,x,\xi),\xi). 
\end{align}
The previous derivation can be also conducted for Maxwellians of the previous type. 
In order to conserve mass it is necessary to assume that $M$ fulfills \eqref{eq:coeff_um1}. 
Then, we obtain a gPC expansion in coefficients $\Bar{\rho}=(\Bar{\rho}_i)_{i=0}^K$ 
and  $\Bar{q}=(\Bar{q}_i)_{i=0}^K$ as 

\begin{subequations} \label{eq:UQBGK}
  \begin{align}
   &   \partial_t \Bar{\rho}_i(t,x) + \partial_x \left[\Bar{q}_i(t,x)- (\P(\Bar{\rho}(t,x))\Bar{\rho}(t,x))_i\right]=0\\
   &   \partial_t \Bar{q}_i(t,x) + \partial_x \left[ \int_W w^2 \Bar{g}_i(t,x,w)\, dw - (\P(\Bar{\rho}(t,x))\Bar{q}(t,x))_i \right]= 
	 \frac1\epsilon \Bigl( \int_W w\overline{M}_i \, dw - \Bar{q}_i(t,x) \Bigr) \\
   & \Bar{\rho}_i(0,x) = \Tilde{\rho}_i(0,x) \; 
 \Bar{q}(0,x) = \Tilde{q}_i(0,x),\; 
 \overline{M}_i = \int_\Omega M(w;\rho(t,x,\xi),\xi) \phi_i(\xi) f_\Xi(\xi) d\xi.
 \end{align}\end{subequations}
Clearly, applying assumption \eqref{eq:ass_thm1} leads for the transport part of the system the same flux as for the Aw--Rascle--Zhang system. A direct identification of the source term with fluid dynamic quantities is no longer possible but depends on the precise dependence of $M$ on $\rho$ and $\xi.$ 
\end{remark} 

\subsection{Stability analysis}\label{sec:mu}

We extend the stability analysis \cite[Section 3.2]{MR4063909} to the stochastic case. Recall,  the stochastic PDE \eqref{eq:BGK-stoch} for the random field $g=g(t,x,w,\xi)$ is given by 

\begin{align}\label{eq:BGK-stoch2}
	&	\partial_t g(t,x,w,\xi)+\partial_x\Big[ (w-h(\rho(t,x,\xi))) g(t,x,w,\xi)\Big] = \frac{1}{\eps} \Big( M_g(w;\rho) - g(t,x,w,\xi) \Big), 
\end{align} 
 where we assume that $M_g$ fulfills \eqref{eq:m1} and \eqref{eq:m2}. The stability analysis in \cite{MR4063909} is based on a Chapman Enskog expansion of $g$ in terms of $\epsilon.$ Here, we similarly assume that 
 
 \begin{align}
 g(t,x,w,\xi)=\Mg(w;\rho(t,x,\xi))+\varepsilon g_1(t,x,w,\xi),
 \end{align}
where a.e. $(x,w)$ and a.s. in $\xi$ 

\begin{align}
	\int_W g_1(t,x,w,\xi) dw = 0.
\end{align}
  Up to terms of order $\epsilon^2$ the perturbation $g_1$ fulfills 
  
 \begin{align}\label{g1}
 	g_1(t,x,w,\xi) = - \partial_t M_g(w;\rho(t,x,\xi)) - \partial_x \left( w - h(\rho(t,x,\xi))
 	\right) M_g(w;\rho(t,x,\xi)).
 \end{align} 
The formal computations presented in \cite[Section 3.2]{MR4063909} extend to the 
above equations \eqref{eq:BGK-stoch2} and \eqref{g1} since they only rely on integration with respect to $w$ and the properties \eqref{eq:m1} and \eqref{eq:m2}. Those are  independent of $\xi$.   Hence, after integrating \eqref{eq:BGK-stoch2} with respect to $w$, substituting $g_1$ by \eqref{g1} as well as subsequent differentiation leads to 

\begin{align}
&\partial_t \rho + 
\partial_x \left( \rho V_{eq}(\rho) \right) = \epsilon 
\partial_x \left( - D(\rho) \partial_x \rho + \partial_x \int_W w^2 M_g(w,\rho) dw 
\right), \label{eq01} \\
& D(\rho) = \left( \partial_\rho Q_{eq}(\rho) + \partial_\rho ( h(\rho) \rho ) 
\right) \left(  \partial_\rho Q_{eq}(\rho) + h(\rho)
\right) + \partial_\rho h(\rho) \left( Q_{eq}(\rho) + h(\rho)\rho
\right) \\
& Q_{eq}(\rho)  = V_{eq}(\rho) \rho, 
\end{align}
where $\rho=\rho(t,x,\xi)$. In \cite{MR4063909} it is assumed that the Maxwellian $M_f$, see Section 2, and $M_g$ are related. In this case

\begin{align}\label{MfMg}
	M_f(v,\rho):=M_g(v+h(\rho),\rho) \; \forall v \in V, \rho \geq 0,
\end{align}
 where $M_f$ is a Maxwellian such that $\int_V M_f(v,\rho) dv = \rho$ and 
 $\int_V v M_f(v,\rho) dv = Q_{eq}(\rho).$ Using those properties equation \eqref{eq01} simplifies and we obtain 
 
 \begin{align}
 		&\partial_t \rho + 
 		\partial_x \left( \rho V_{eq}(\rho) \right) = \epsilon 
 		\partial_x \left(  \mu(\rho) \partial_x \rho \right), \label{eq02} \\
 		& \mu(\rho) = \left(  - \partial_\rho Q_{eq}(\rho)^2 - \partial_\rho  h(\rho) \partial_\rho Q_{eq}(\rho) \rho + Q_{eq}(\rho) \partial_\rho h(\rho)     
 		\right) + \int_V v^2 \partial_\rho M_f(v,\rho) dv. \label{eq_mu}
 \end{align}
Note that in the presented case $\mu$ is in fact a random field through its dependence on 
$\rho=\rho(t,x,\xi).$ Therefore, compared with the deterministic case, we may now infer information on e.g. expectation, confidence bands or the probability of $\mu$ to be non-positive. In particular, the later is relevant for qualitative assessment of traffic flow since it shows where possible instabilities may occur. Hence,  for fixed $t\geq 0$ and $x\in \R$ consider 

\begin{align}\label{Pmu}
	\mathbb{P}_{t,x}(\mu \leq 0 ) := \int_\Omega H(-\mu(\rho(t,x,\xi)) f_\Xi(\xi) d\xi.
\end{align}
 It has been argued  in \cite{MR4063909} that \eqref{Pmu} indicates regions of traffic situations of high risk. Further, points  $(t,x)$ where $\mathbb{P}_{t,x}(\cdot) >0$ holds might lead to the rise of stop--and--go waves.  A numerical investigation of \eqref{Pmu} will be presented in the forthcoming section. 
\par 
Note that the computation of \eqref{Pmu} requires to reconstruct the stochastic field $\rho(t,x,\xi)$. This can be obtained by reconstruction of $g$ given by \eqref{eq:gPC1} where
$\Tilde{g}$ are given by equation \eqref{eq:UBGK}.

For particular choices of $V_{eq}(\cdot)$ and $h(\cdot)$ the gPC expansion of the first terms in $\mu$ can be obtained directly using the moments $\Tilde{\rho}.$ In fact, assume $h(\rho)=\rho$ and $V_{eq}=\rho_{max} -\rho.$  Then, $Q_{eq}=\rho(\rho_{max}-\rho)$ is the flux of the Lighthill-Whitham--Richards model and equation \eqref{eq02} simplifies

\begin{align}
	&	\mu(\rho) = R(\rho)+\int_V v^2 \partial_\rho M_f(v,\rho) dv,\\
	& R(\rho) :=   - (\rho_{max}-2\rho)^2 - (\rho_{max} \rho -\rho^2)  + \rho (\rho_{max}-\rho). 
\end{align}
Hence, we obtain 

\begin{align} 
	G_K( R(\rho(t,x,\cdot)))(\xi) = \sum_{i=0}^K \Tilde{R}_i(t,x) \phi_i(\xi)
\end{align}
where $\Tilde{R}$ is expressed in terms of $\Tilde{\rho}$ and $\mathbb{1}=(1,\dots,1)^T \in \R^{K+1}$ 

\begin{align}
	\Tilde{R}(t,x) = - P(\rho_{max} \mathbb{1} -2 \tilde{\rho} )(\rho_{max} \mathbb{1} - 2 \Tilde{\rho}) - \rho_{max} \Tilde{\rho} + P(\Tilde{\rho})\Tilde{\rho} + P(\Tilde{\rho})(\rho_{max}\mathbb{1}-\Tilde{\rho})
\end{align}
However, in  the numerical simulations we use an Maxwellian $M_f$ obtained by a discrete velocity model, see below for the details. Therein,  $\rho$ enters within a rational polynomial and a simple expression as above also for the expansion of $\partial_\rho M_f(v,\rho) dv$ seems to be not possible.

\section{Computation results}
 Numerically, we are interested in indicating and forecast,  regions of high risk of congestion or instabilities. For this reason we focus on the simulation of \eqref{Pmu}. First,  we perform a steady state analysis and investigate parameters influencing regions of high probability. 
Secondly, we show the evolution of this probability in time. 
\par 
As Maxwellian we choose a discrete velocity distribution with $N$ velocities as in  \cite{MR4063909}:

\begin{equation}
	M_f(v;\rho)=\sum_{j=1}^N f^{\infty}_j(\rho)\delta_{v_j}(v).
\end{equation}
The weights are normalized to ensure equation \eqref{eq:um1}, i.e.,  $\sum_{j=1}^N f^{\infty}_j(\rho)=\rho$ for any $\rho>0$. The set of velocities is $\{v_j\}_{j=1}^N$. Then, for fixed $\rho>0$ the weights are recursively defined by 

\begin{align}
	f^\infty_j=&   \begin{cases}
		0 & \rho \ge\frac{1}{2} \\
		\frac{-2(1-\rho )\sum_{k=1}^{j-1} f^{\infty}_k+(1-2\rho )\rho+\sqrt{[(1-2\rho )\rho -2(1-\rho )\sum_{k=1}^{j-1}f^{\infty}_k ]^2 +4\rho (1-\rho )\rho f^{\infty}_k}}{2(1-\rho )} & else
	\end{cases}, \\& j=1,\dots, N - 1, \nonumber \\
f^\infty_N=&\rho-\sum_{j=1}^{N-1}f^\infty_j. 
\end{align}
We use $V_{eq}=1-2 \rho$ and $h(\rho)$ as indicated in the tests below. The Maxwellian $M_g$ is obtained through relation \eqref{MfMg}. Since the previous Maxwellian is a rational polynomial of $\rho,$ an explicit expression of $\mu(\cdot)$ in terms of the moments of $\rho$ might not be feasible. Therefore, we evaluate $\eqref{Pmu}$ numerically using quadrature with $N_\xi$ number of points. 
\par 
The gPC Aw--Rascle--Zhang system is discretized as in \cite{gerster2021stability}, i.e., employing a  local Lax Friedrichs scheme to solve \eqref{eq:UARZ}. 
\par 
The numerical parameters are as follows.  We consider the space interval $x\in [a,b]=[0,2]$ and define the uniform spatial grid of size  $\Delta x=2\cdot 10^{-2}$. Moreover, let $T_f=1$ be the final time of the simulations and $\Delta t$ the time step, which is chosen in such a way that the CFL condition is fulfilled. By $N_t$ we denote the number of the time steps needed to reach $T_f$. The random variable $\omega$ is assumed be uniform distributed on $(0,1)$, i.e., $f_\Xi=1$ and $\Omega=(0,1).$ As basis functions we consider the Haar basis, which are known to fulfill (A1)--(A3). The numerical quadrature of \eqref{Pmu} is conducted with a uniform discretization of $(0,1)$ in $\xi$ with  $N_\xi=10^4$  quadrature nodes. Whenever necessary the random fields   density and flux are   approximated  up to a specified order $K$ by  $\rho(t,x,\xi)= \sum_{i=0}^{K} \Tilde{\rho}_i(t,x) \phi_i(\xi)$ and similarly for $q(t,x,\xi).$

\subsection{Steady state analysis}

The Maxwellian $M_f$ depends on two parameters, the number of discrete velocities $n_v$ governing the level of description of traffic as well as the local density $0<\rho<1.$ In the steady state case the density $\rho$ has been a constant in the deterministic case \cite{MR4063909}, however,  it is now  a random parameter. Since we are interested in the stability of traffic patterns we setup the steady state problem as follows: We assume a constant traffic density $\rho_0$ that is perturb by (possibly small) perturbation

\begin{align}\label{rhoxi}
	\rho(\xi) = \rho_0 + \sigma (\xi-\frac12), 
\end{align}
where $\sigma>0$ controls the standard deviation and the factor $\frac12$ is included to have zero mean for $\xi$ uniformly distributed.  We are interested in the probability $\eqref{Pmu}$ for the previous choice of $\rho$ and $\rho_0 \in [0.1; 0.9]$ with $\sigma=0.1$. The resulting $\mathbb{P}(\mu\leq0)$ is shown in Figure \ref{fig:test2} (red curve). 
\par 
 In the free flow regime, i.e., $\rho_0<\frac12$, the probability of instabilities is zero,  it is increasing until its maximum transition regime, and decreasing in the congested are $\rho>\frac12$. It is interesting to observe that  in the congested region the probability of $\mu<0$ is close to zero with the interpretation that the traffic propagates at  low speed and  and no free space to accelerate. As expected, the highest probability for instabilities is  in the transitional regime $\rho_0 \approx (0.5, 0.6)$. 
\par 
Moreover, in Figure \ref{fig:test2}, we compare the predictions for $n_v=3$ and $n_v=10$. In case of a small and large number of discrete velocities. In case of only three velocities the transition region stretches up to the maximal density due to the limit choices of velocities the drivers can attain. For $n_v=10$, we observe the highest probability for $\rho$ between $0.5$ and $0.6$ as detailed above. 

\begin{figure}[h!]
\centering
	\includegraphics[width=0.62\textwidth]{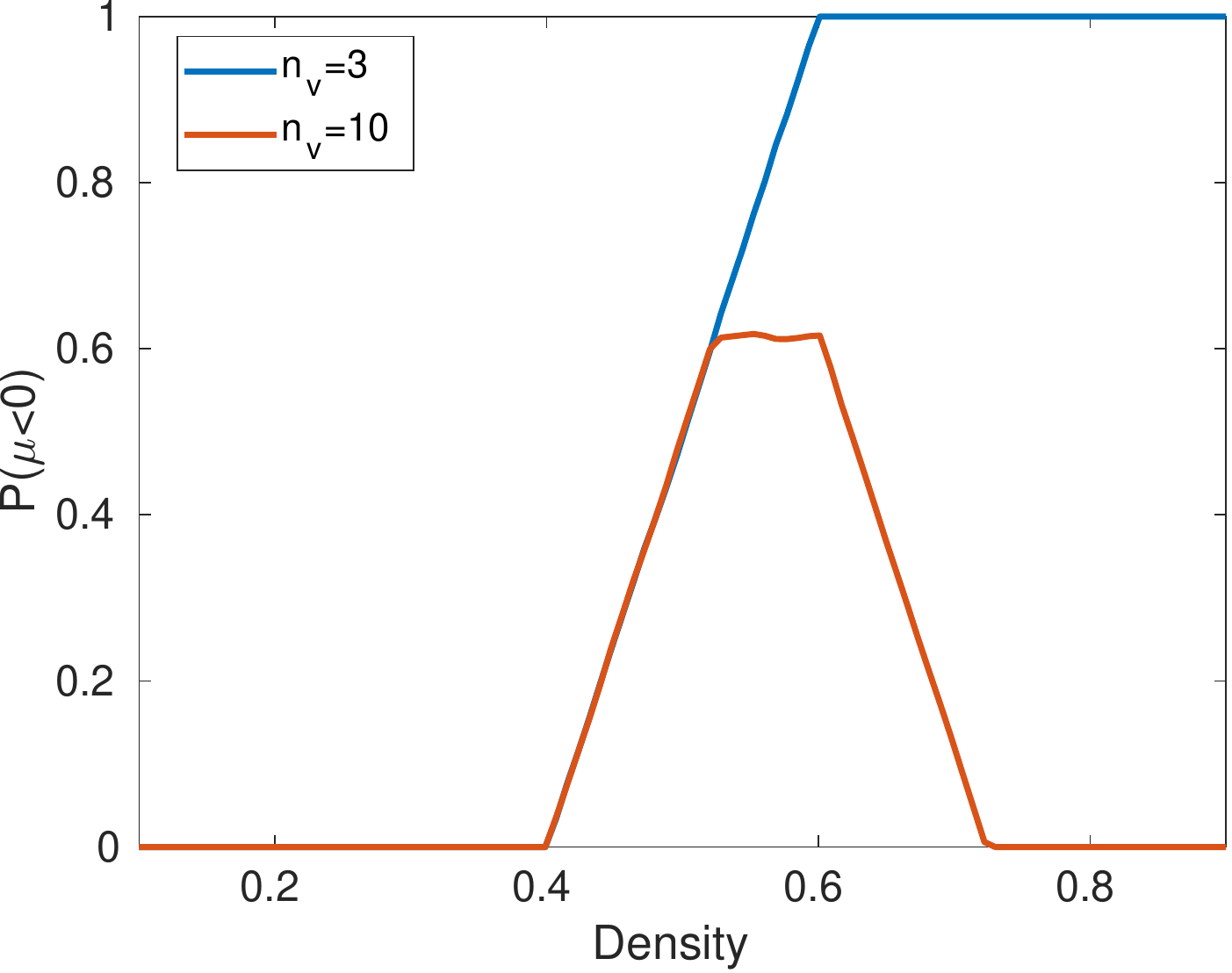}
	\caption{Probability \eqref{Pmu} for a Maxwellian $M_f$ with different number of discrete velocities: $n_V=3$, $n_V=10$. On the x-axis $\rho_0$ is shown, see \eqref{rhoxi}. }
	\label{fig:test2}
\end{figure}

Further, the dependence of  $\mathbb{P}(\mu<0)$ for fixed values of $\rho_0$ but different standard deviations $\sigma$ is shown. We observe a different behavior if we start from $\rho_0=0.4$ and $\rho_0=0.6$, see Figure \ref{fig:n_var_uni}. In the latter, the probability is decreasing with the possible explanation that  the density is spreading far from the  transition area. In the former, we are in the opposite situation, since we  approach the  transition region for increasing value for the standard deviation.  

\begin{figure}[h!]
	\includegraphics[width=0.48\textwidth]{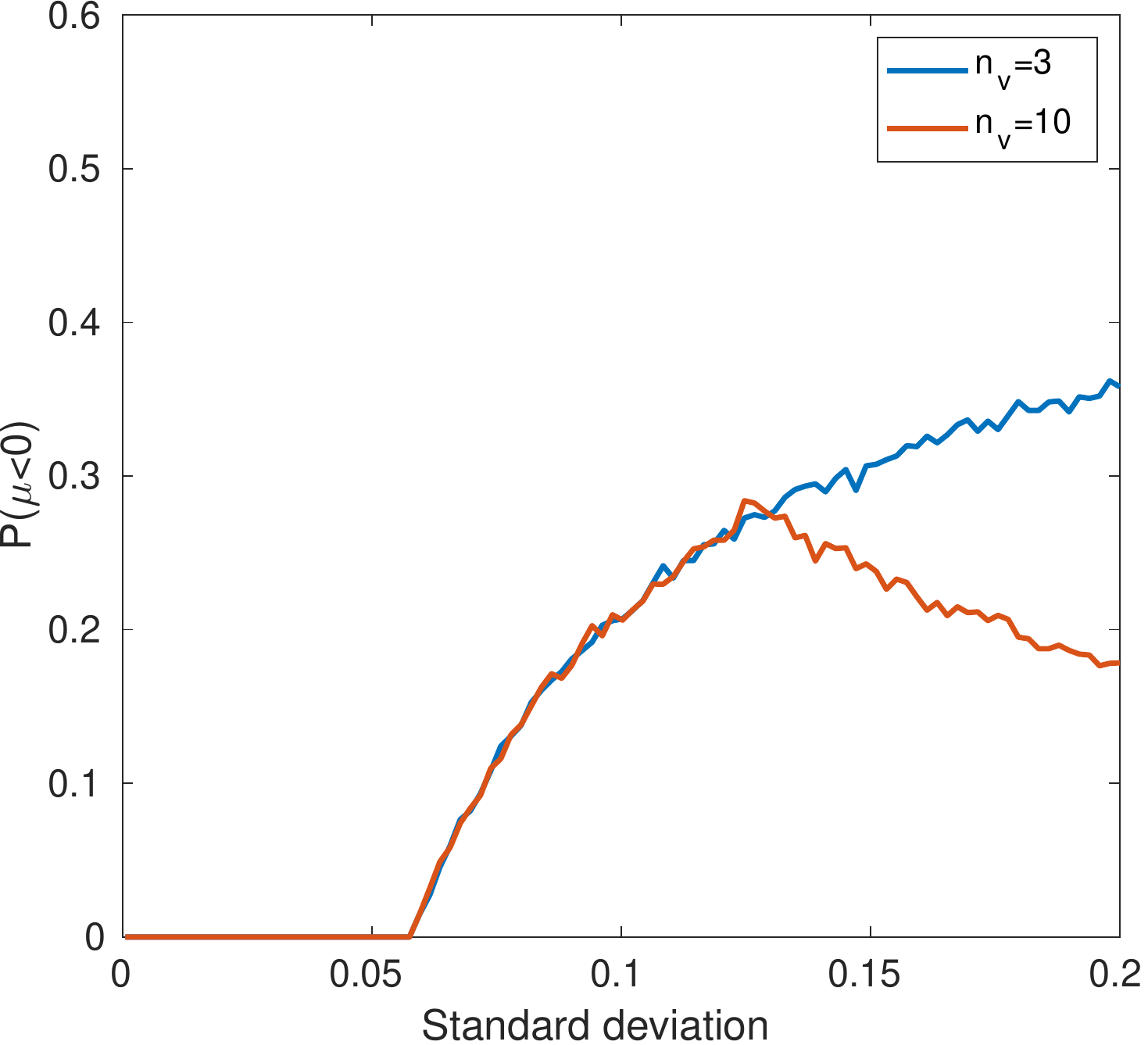}
	\includegraphics[width=0.52\textwidth]{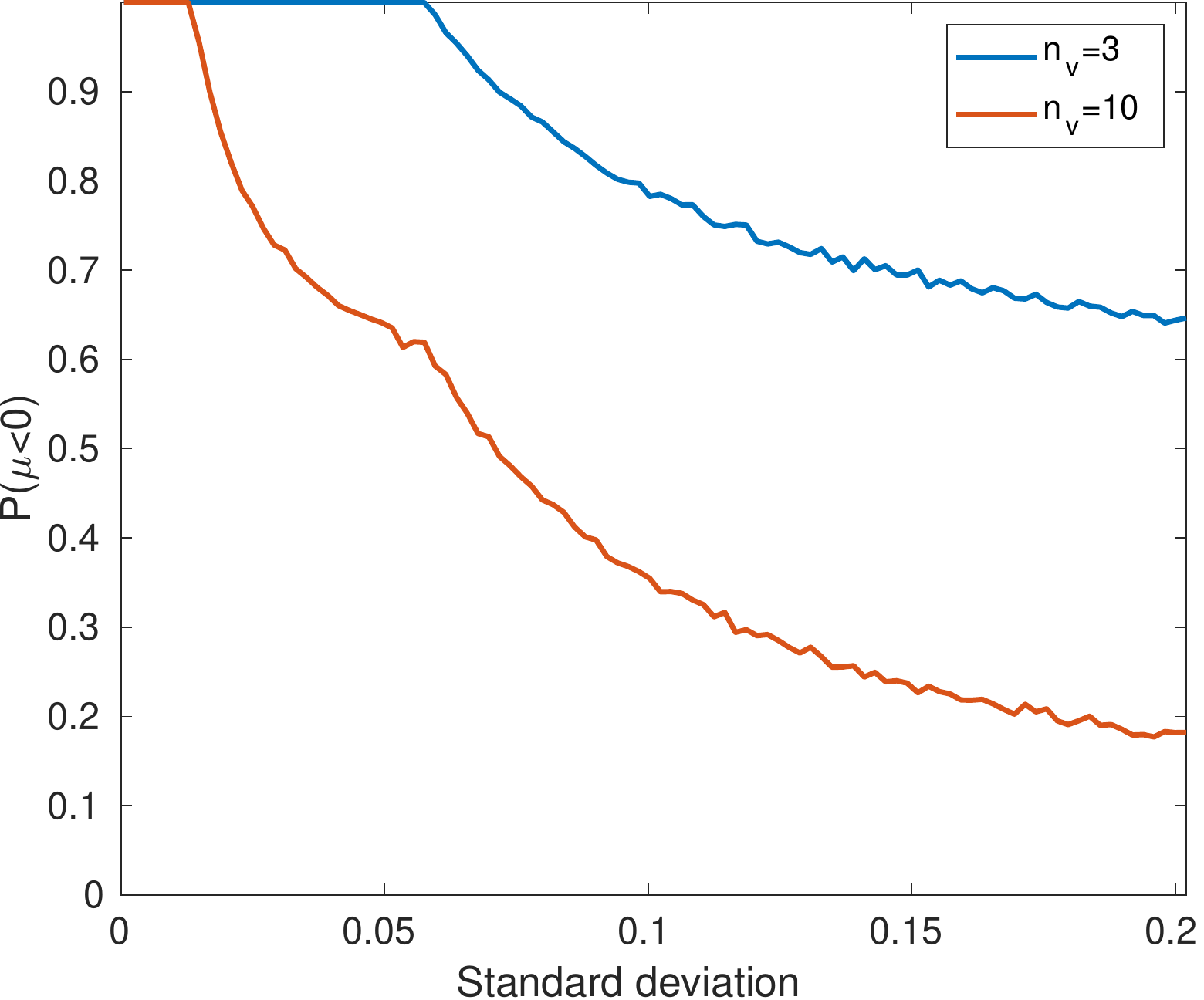}
	\caption{Probability \eqref{Pmu}  for  different velocities samples: $n_v=3$ (blue line), $n_V=10$(red line) for different values of  $\rho_0=0.4$(left)  $\rho_0=0.6$(right) when the standard deviation $\sigma$ ranges  from zero to $0.2$.}
	\label{fig:n_var_uni}
\end{figure}

Furthermore,  we  compare also the effect of different hesitation functions  In Figure \ref{fig:diff_press}, $h(\rho)=\rho$(blue line) and $h(\rho)=\rho^3$ are considered. The observed behavior is very similar. 

\begin{figure}[h!]
\centering
	\includegraphics[width=0.53\textwidth]{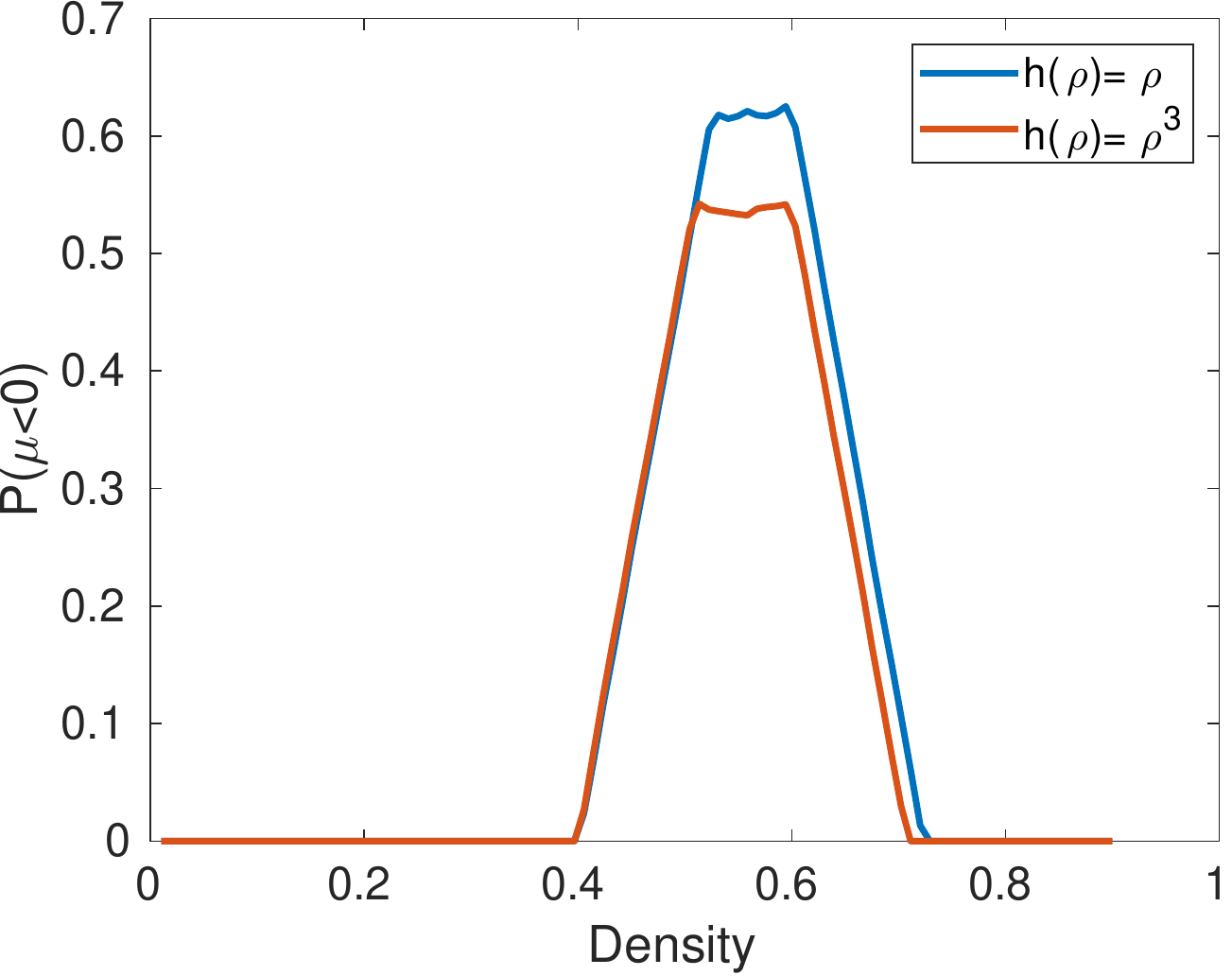}
	\caption{Probability \eqref{Pmu} for  different hesitation functions: $h(\rho)=\rho$(blue line) and $h(\rho)=\rho^3$(red line).}
	\label{fig:diff_press}
\end{figure}


\subsection{Time--dependent problems}

We investigate numerically if the dynamics amplifies the probability of instabilities starting from a Riemann problem. In order to evaluate \eqref{Pmu} for a temporal and spatially dependent $\rho$ we need to reconstruct the density and therefore we first show the convergence in $K.$ 
In all following computations we set $n_v=5$, $N_\xi= 10^2$ and define the Riemann problem:

\begin{equation}\label{eq:initial_data_rar}
\rho(x,0,\xi)=\begin{cases}
\rho_l\equiv \xi \sim \mathcal{U}(0.55, 0.85)   \qquad & x<1\\
0.2 & x\ge 1
\end{cases},
\qquad
v(x,0,\xi)=\begin{cases}
0.2   \qquad & x<1\\
0.7 & x\ge 1
\end{cases}.
\end{equation}

In Figure \ref{fig:diff_K} we show  $\mathbb{P}_{T_f,x}(\mu<0)$ for an increasing number of base functions $K$. Moreover, the probability of instabilities is increasing in time and travels backward. As explanation of this behavior we note that the the given data leads to a rarefaction wave in $(t,x)$ for any fixed $\rho_l.$ Hence, drivers observe a free flow area ahead and accelerate. The $95\%-$confidence band of the density $\rho$ at the final time $T_f$ is shown in Figure \ref{fig:rar}(right).

\begin{figure}[h!]
\includegraphics[width=0.54\textwidth]{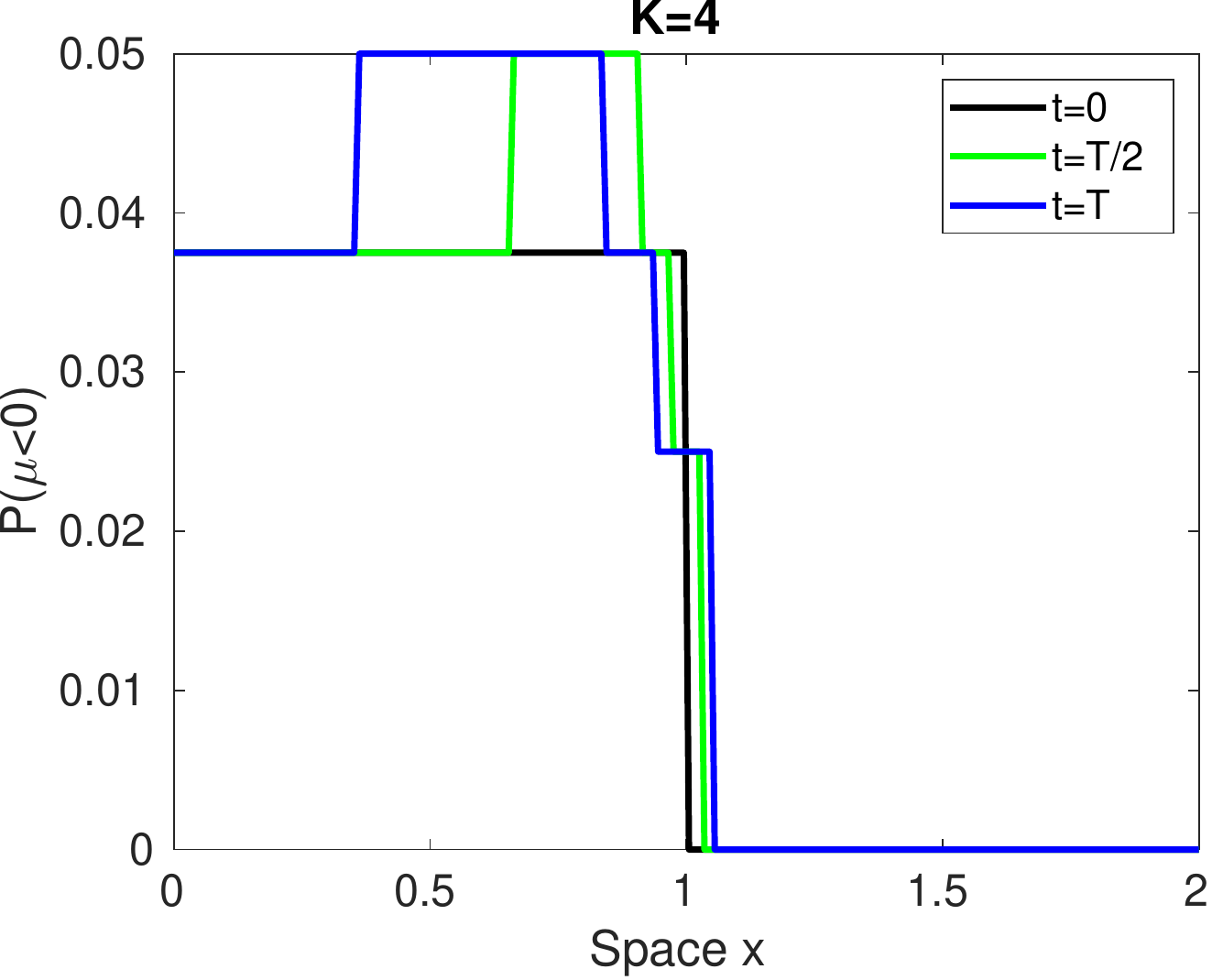}
\includegraphics[width=0.54\textwidth]{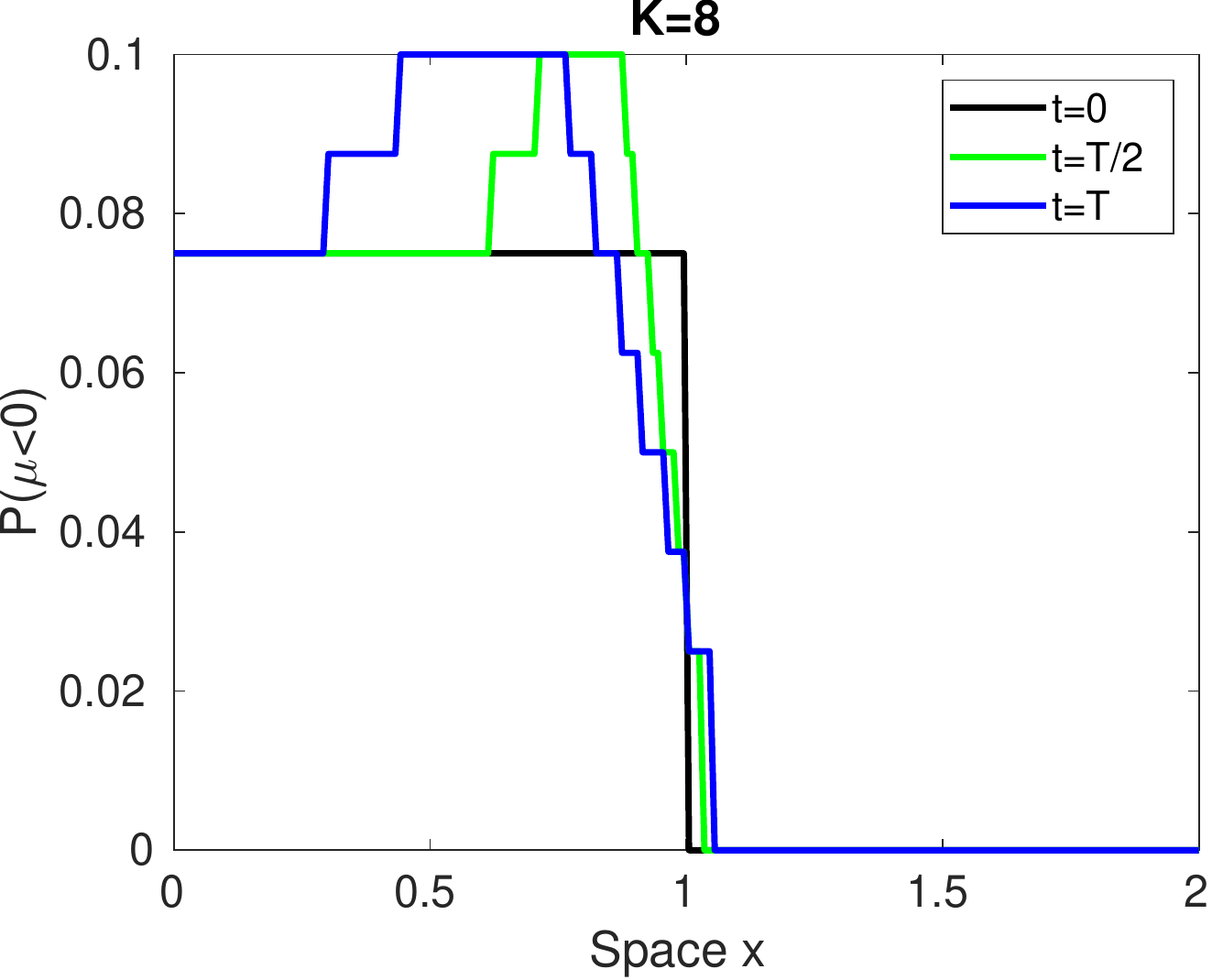}
\includegraphics[width=0.54\textwidth]{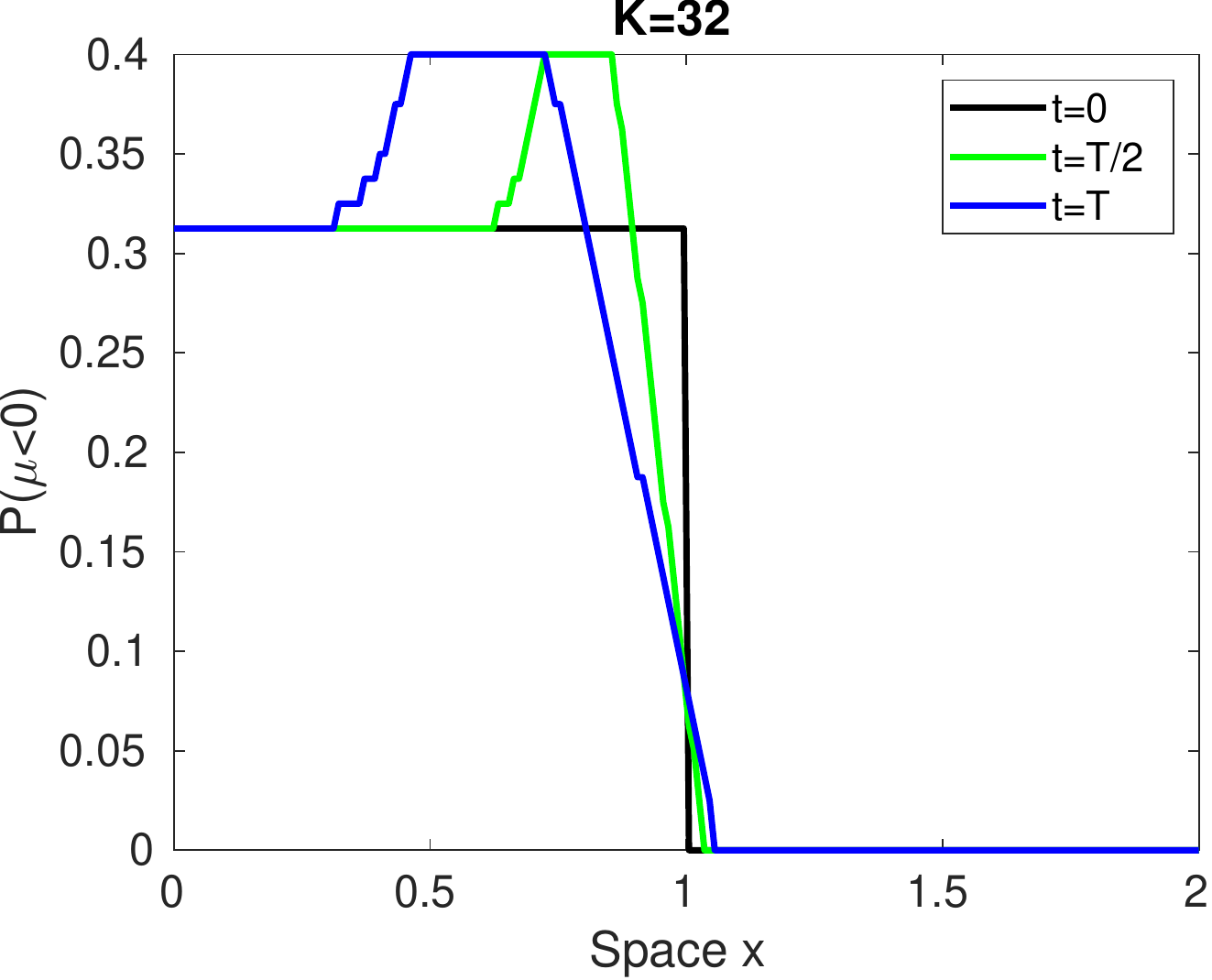}
\includegraphics[width=0.52\textwidth]{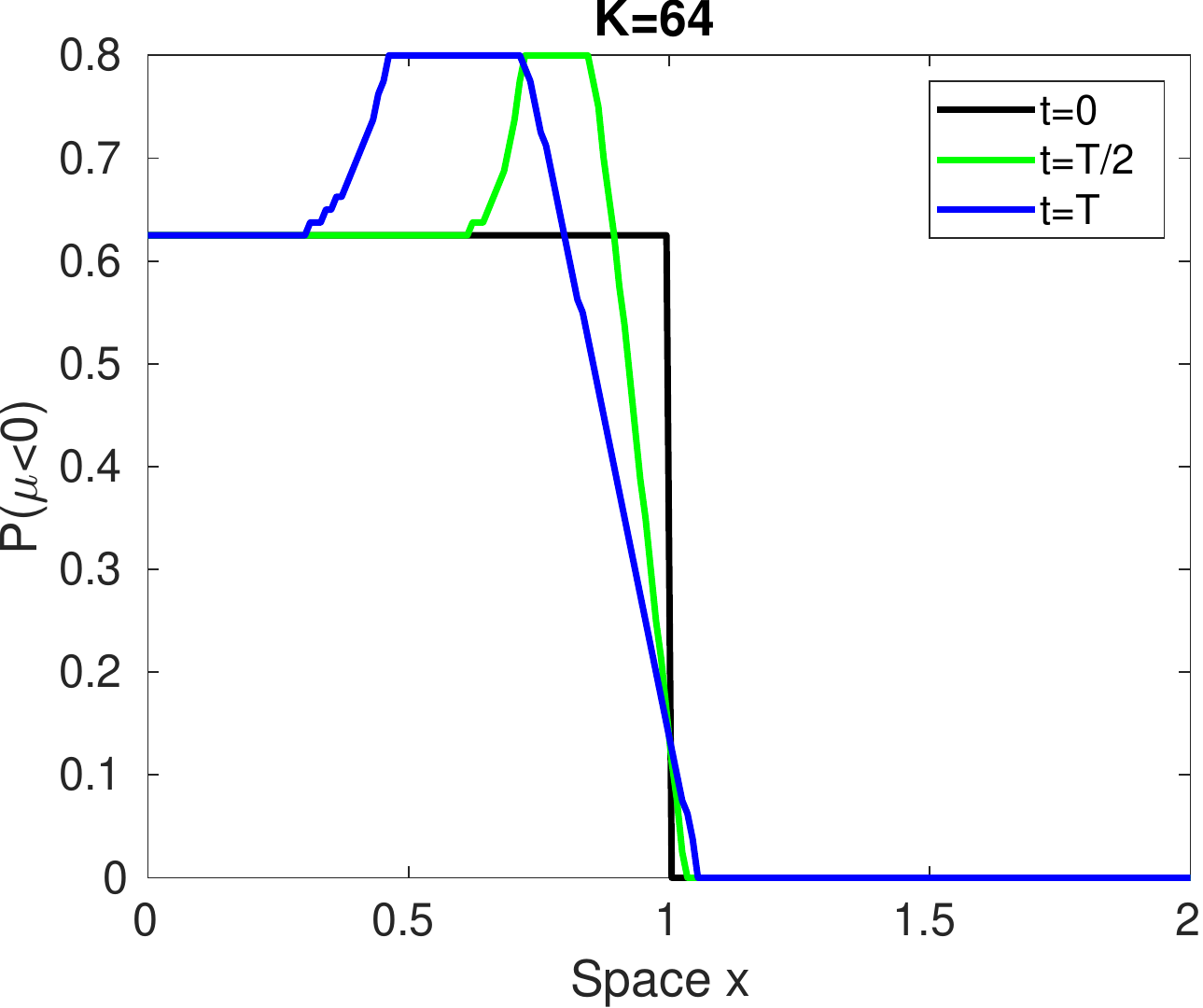}
    \caption{Probability \eqref{Pmu} at time $t=0$ (black line), $t=\frac{T:f}{2}$ (green line) $t=T$ (blue line) for $K=4$ (top-left), $K=8$ (top-right), $K=32$ (bottom-left), $K=64$ (bottom-right).}
    \label{fig:diff_K}
\end{figure}


\begin{figure}[h!]
\includegraphics[width=0.55\textwidth]{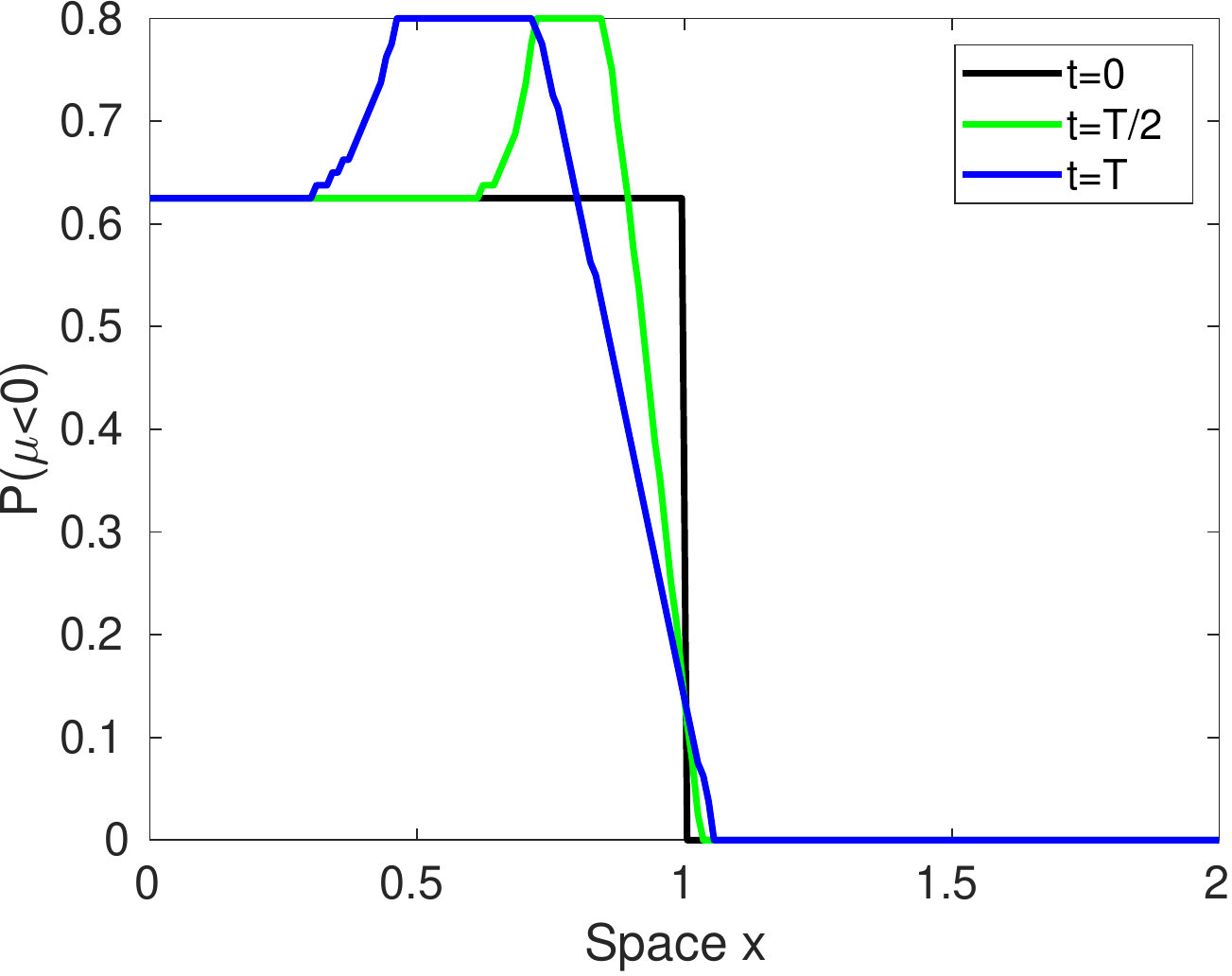}
\includegraphics[width=0.5\textwidth]{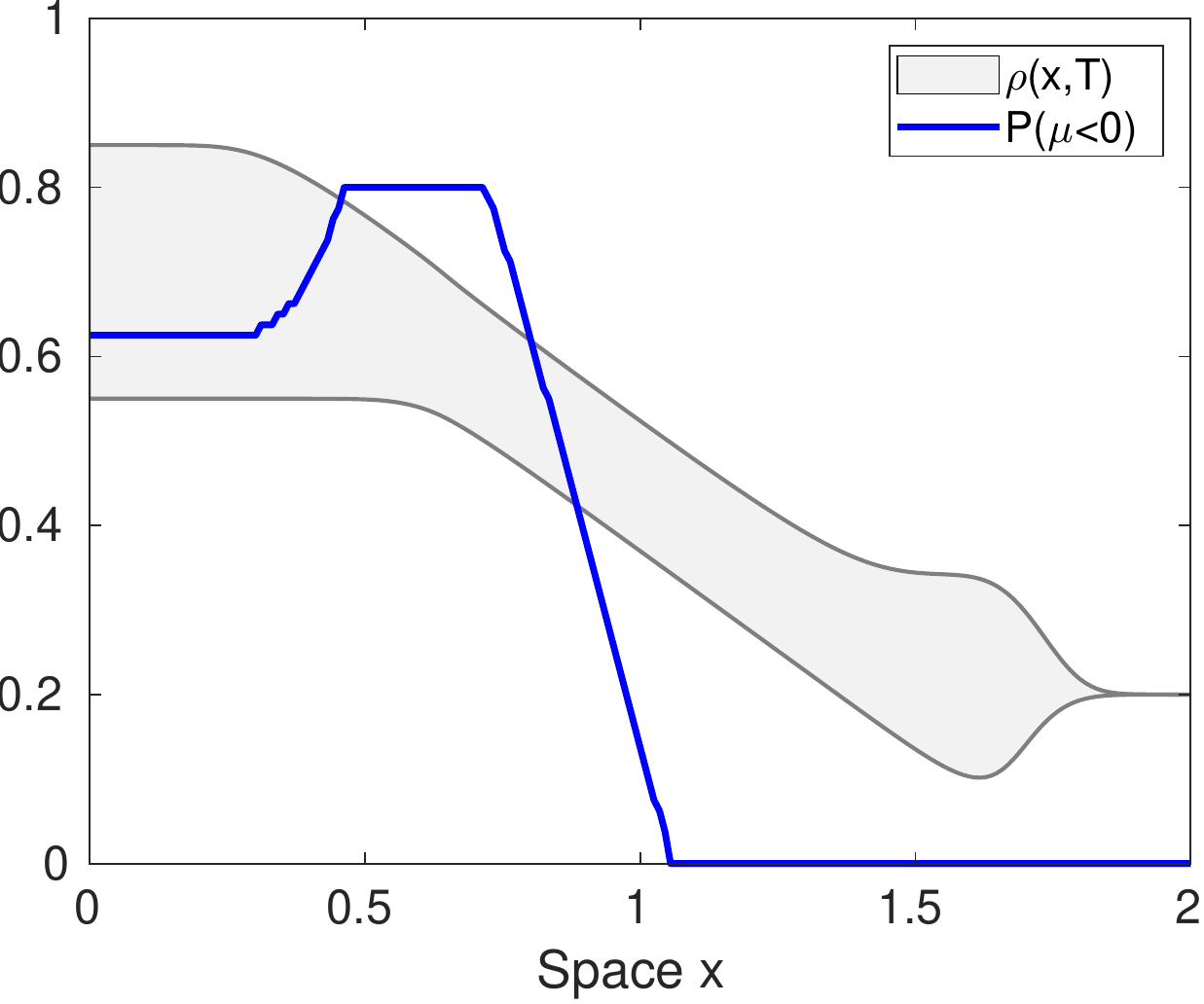}
\caption{Probability of negative diffusion coefficient in a rarefaction case at different time: $t=0$, $t=\frac{T_f}{2}$,$t=T_f$, $K=64$, and comparison with the confident region of the density at $t=T_f$.}
    \label{fig:rar}
\end{figure}

As second test case we consider a (random) shock wave as given by \eqref{eq:initial_data_shock}. Here,  the probability of instabilities increases and spreads both forward and backward. A possible explanation might be that the  vehicles have to decelerate in order to avoid collisions leading also to backwards spreading waves.  
\begin{equation}\label{eq:initial_data_shock}
\rho(x,0,\xi)=\begin{cases}
\rho_l\equiv\xi \sim \mathcal{U}(0.15, 0.45)   \qquad & x<1\\
0.75 & x\ge 1
\end{cases},
\qquad
v(x,0,\xi)=\begin{cases}
0.7   \qquad & x<1\\
0.3 & x\ge 1
\end{cases}.
\end{equation}
In Figure \ref{fig:shock}(right) the $95\%-$confidence band of the density at the final time $T_f$ is shown. It is interesting to note that even starting from $\mathbb{P}(\mu<0) \equiv 0$ does not ensure to avoid instabilities as  time evolves, see Figure \ref{fig:shock}(left). Indeed, at time $t=0$ the probability is zero. However, as the time evolves, the density enters the  transition phase and the probability of instability grows. 

\begin{figure}[h!]
\includegraphics[width=0.55\textwidth]{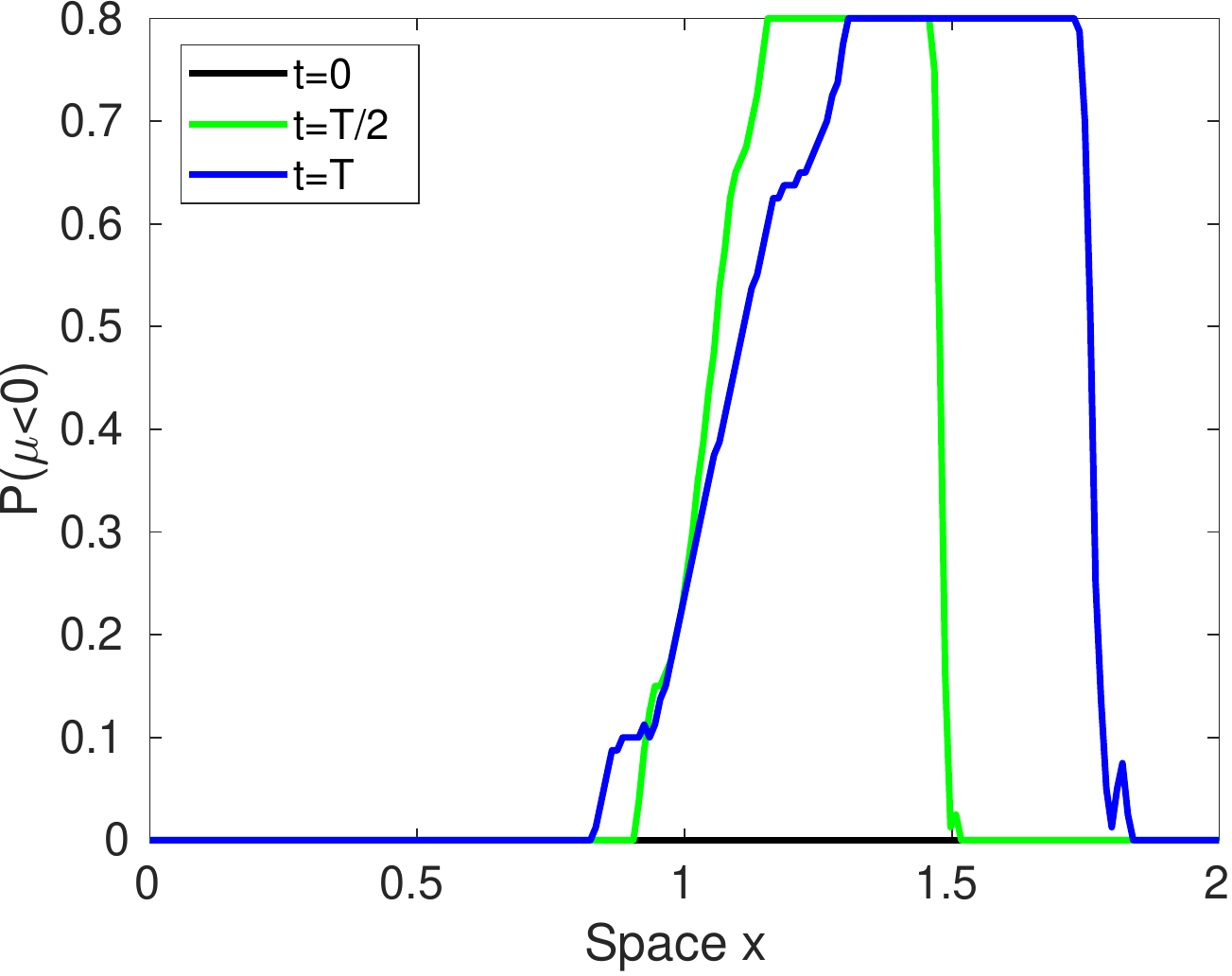}
\includegraphics[width=0.52\textwidth]{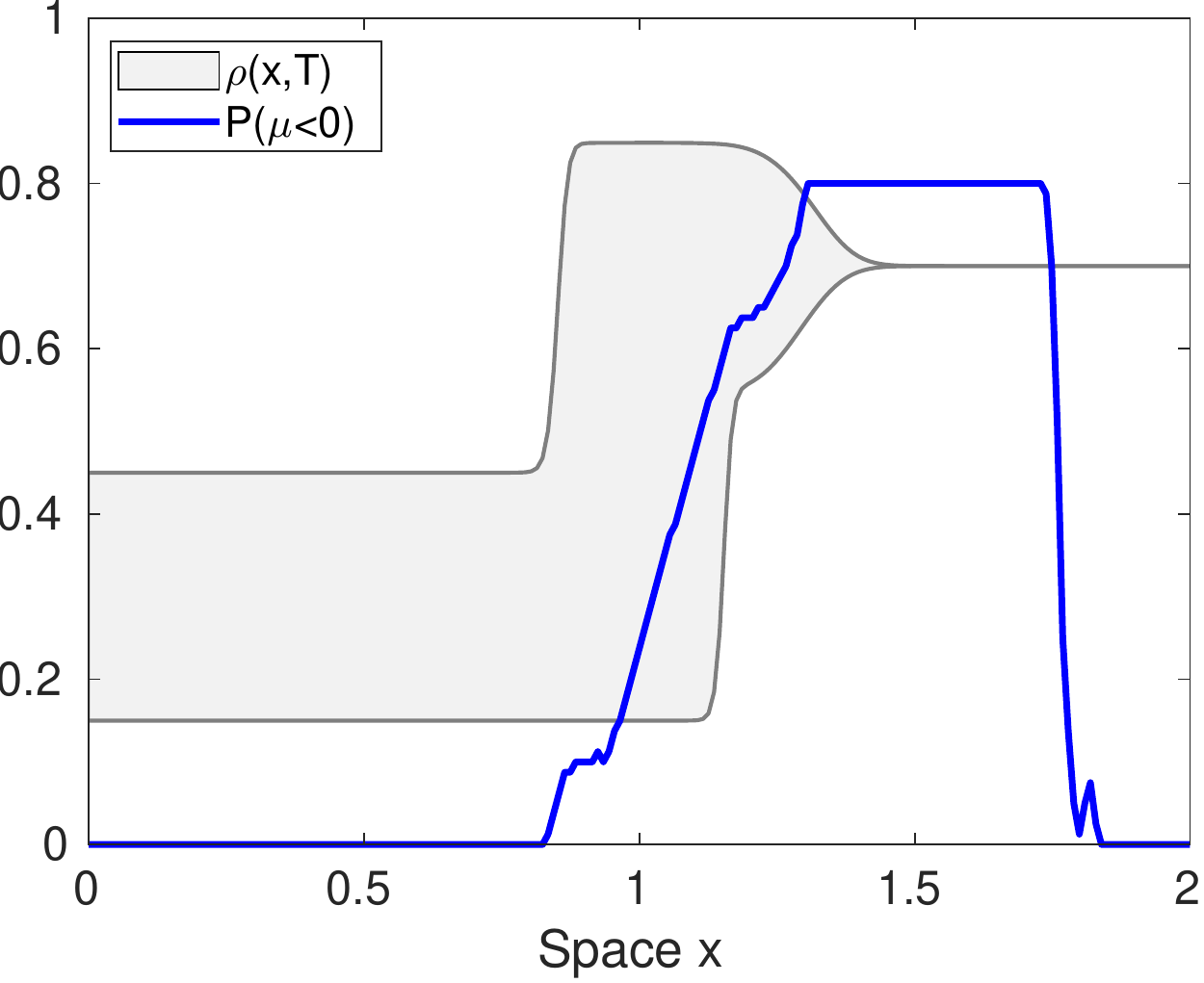}
    \caption{Density profile and probability of negative diffusion coefficient in a shock case at $t=T_f$, $K=64$.}
    \label{fig:shock}
\end{figure}

\subsection*{Acknowledgments}
The authors thank the Deutsche Forschungsgemeinschaft (DFG, German Research Foundation) for the financial support through 20021702/GRK2326,  333849990/IRTG-2379, HE5386/19-1,22-1,23-1 and under Germany's Excellence Strategy EXC-2023 Internet of Production 390621612.

\bibliography{ref}
\bibliographystyle{siam}

\end{document}